\documentclass[11pt,a4paper,english]{article}
\usepackage[T1]{fontenc}
\usepackage[latin9]{inputenc}
\usepackage{enumitem}
\usepackage{amsmath}
\usepackage{amsthm}
\usepackage{amssymb}

\usepackage{soul}
\usepackage{xcolor}

\makeatletter

\theoremstyle{plain}
\newtheorem{thm}{\protect\theoremname}[section]
\theoremstyle{plain}
\newtheorem{conjecture}[thm]{\protect\conjecturename}
\theoremstyle{plain}
\newtheorem{lem}[thm]{\protect\lemmaname}
\theoremstyle{definition}
\newtheorem{defn}[thm]{\protect\definitionname}
\theoremstyle{plain}
\newtheorem{prop}[thm]{\protect\propositionname}
\ifx\proof\undefined
\newenvironment{proof}[1][\protect\proofname]{\par
	\normalfont\topsep6\p@\@plus6\p@\relax
	\trivlist
	\itemindent\parindent
	\item[\hskip\labelsep\scshape #1]\ignorespaces
}{%
	\endtrivlist\@endpefalse
}
\providecommand{\proofname}{Proof}
\fi
\theoremstyle{remark}
\newtheorem{rem}[thm]{\protect\remarkname}

\usepackage{authblk}
\date{}
\DeclareMathOperator{\diam}{diam}
\DeclareMathOperator{\spn}{span}
\DeclareMathOperator{\ii}{\textbf{i}}
\DeclareMathOperator{\jj}{\textbf{j}}
\DeclareMathOperator{\uu}{\textbf{u}}
\DeclareMathOperator{\vv}{\textbf{v}}

\DeclareMathOperator{\supp}{supp}
\DeclareMathOperator{\cov}{cov}

\DeclareMathOperator{\rank}{rank}
\newcommand{\oR}{$\textrm{R}^*$}
\newcommand{\uR}{$\textrm{R}_*$}
\makeatother

\usepackage{babel}
\providecommand{\conjecturename}{Conjecture}
\providecommand{\definitionname}{Definition}
\providecommand{\lemmaname}{Lemma}
\providecommand{\propositionname}{Proposition}
\providecommand{\remarkname}{Remark}
\providecommand{\theoremname}{Theorem}

\begin{document}
\title{New results on embeddings of self-similar sets via renormalization}
\author{Amir Algom\thanks{Research supported by Grant No. 2022034 from the United States - Israel Binational Science Foundation  (BSF), Jerusalem, Israel.}, Michael Hochman\thanks{Research supported by ISF research grant 3056/21.}
~and Meng Wu\thanks{Research supported by the Academy of Finland,  project grant No. 318217.}}
\maketitle
\begin{abstract}
For self-similar sets $X,Y\subseteq \mathbb{R}$, we obtain new results towards the affine embeddings conjecture of Feng-Huang-Rao (2014), and the equivalent weak intersections conjecture. We show that the conjecture holds when the defining maps of $X,Y$ have algebraic contraction ratios, 
and also for arbitrary $Y$ when the maps defining $X$ have algebraic- contraction ratios and there are sufficiently many of them relative to the number of maps defining $Y$.
We also show that it holds for a.e. choice of the corresponding contraction ratios, and obtain bounds on the packing dimension of the exceptional parameters.  Key ingredients in our argument include  a new renormalization procedure in the space of embeddings, and Orponen's projection Theorem for Assouad dimension (2021).
\end{abstract}

\section{\label{sec:Introduction}Introduction}

Let $X,Y\subseteq\mathbb{R}$ be non-trivial strongly separated self-similar
sets. That is, 
$$X=\bigcup_{i\in I}\varphi_{i}(X)\;\;,\;\; Y=\bigcup_{j\in J}\psi_{j}(Y)$$
are non-empty compact sets, the unions are disjoint, $I,J$
are finite index sets of size at least two, and $\varphi_{i},\psi_{j}:\mathbb{R}\rightarrow\mathbb{R}$
are non-constant contracting affine maps. Denote the absolute value
of the scaling constant of an affine map $f$  by $\left\Vert f\right\Vert $, and let $\alpha_{i}:=\left\Vert \varphi_{i}\right\Vert $ and $\beta_{j}:=\left\Vert \psi_{j}\right\Vert $. We say that $X$ is homogeneous when all the maps $\varphi_i$ have a common scaling constant. 

Let $\mathcal{A}$ denote the set of non-singular affine maps $\mathbb{R}\to\mathbb{R}$. Consider the set $\mathcal{E}=\mathcal{E}(X,Y)\subseteq\mathcal{A}$ of affine embeddings  of $X$ into $Y$:
\[
\mathcal{E}= \lbrace f\in\mathcal{A}\;:\, f(X)\subseteq Y\rbrace
\]
We are interested in the following conjecture of Feng, Huang, and Rao \cite{FengHuangRui2014}:
\begin{conjecture}
\label{conj:main}If $\log\alpha_{i}\notin\spn_{\mathbb{Q}}\{\log\beta_{j}\}_{j\in J}$
for some $i\in I$, then $\mathcal{E}=\emptyset$.
\end{conjecture}

The existence of $i\in I$ such that  $\log \alpha_i\notin\spn_{\mathbb{Q}}\{\log\beta_{j}\}$ is a property of the sets, i.e, it is independent of the choice of affine maps defining $X$ and $Y$. This follows from the so-called logarithmic commensurability principle; see \cite[Theorem 1.1]{feng2009structures} and \cite[Theorem 4.9]{ElekesKeletiMathe2010} for the general case. 


Conjecture \ref{conj:main} is closely related to questions	 about the dimension of the intersection $X\cap Y$. One of the central results in this area is Furstenberg's intersections conjecture, which was recently verified by Shmerkin \cite{Shmerkin2019} and Wu \cite{Wu2019}, see also \cite{austin2020new,Yu2021imrov}. Assuming that $X$ and $Y$ are homogeneous with uniform scaling
constants $\alpha,\beta$ respectively, Furstenberg's conjecture asserts  that
$$\frac{\log\beta}{\log\alpha}\notin\mathbb{Q} \;\;\Rightarrow\;\; \dim X\cap Y\leq\max\{0,\dim X+\dim Y-1\}.$$ 
This implies Conjecture \ref{conj:main} when $X,Y$ are homogeneous. Indeed, the algebraic hypotheses coincide in this case, and if $f\in\mathcal{A}$ then
$f(X)$ is self-similar with the same scaling constants 
as $X$, so the inequality above implies that $\dim f(X)\cap Y<\dim f(X)$, hence $f\notin \mathcal{E}$. This shows that $\mathcal{E}=\emptyset$.

The extension of Furstenberg's conjecture to the non-homogeneous case is entirely open.  It might be expected to hold under the natural arithmetical condition, namely when $\spn_\mathbb{Q}\{\log\alpha_i\}\cap\spn_\mathbb{Q}\{\log\beta_j\}=\{0\}$, but it is not know if under this condition we even get any dimension reduction, i.e.\ whether \[
\dim X\cap Y<\min\{\dim X,\dim Y\}.
\]
But it turns out that the last inequality is equivalent to $\mathcal{E}(X,Y)=\mathcal{E}(Y,X)=\emptyset$ \cite{FengHuangRui2014}. Thus,  Conjecture \ref{conj:main} can be viewed as a weak form of the intersections conjecture for non-homogeneous self-similar sets.

At present, in the non-homogeneous case, Conjecture \ref{conj:main} is known in two main cases. First, when $\dim Y$ is small enough, specifically, when $\dim Y<1/(2d+2)$, where $d=\rank_\mathbb{Q}\{\beta_j\}_{j\in J}\cup\{1\}$ or $\dim Y<1/4$ when $|J|=2$  \cite{FengXiong2018,Yu2019}. The second case is when $\dim Y-\dim X$ is small enough in a manner depending on
$\dim Y$ \cite{Algom2018, Algom2020}. It is also known for small values of $\dim Y$ under a combination of Diophantine and Liouvillian assumptions on the scaling constants \cite{Baker2021}.

In this paper we prove several new results towards the Conjecture. 

\begin{thm}\label{thm:algebraic}
Conjecture \ref{conj:main} holds when all the scaling constants $\alpha_{i},\beta_{j}$ are algebraic.
\end{thm}

Since the hypothesis of Conjecture \ref{conj:main} holds when $\rank_\mathbb{Q}\{\log \alpha_i\}>\rank_\mathbb{Q}\{\log \beta_j\}$,  in this case we should expect that  $\mathcal{E}=\emptyset$. Indeed, in this situation there certainly will exist an $i$ satisfying $\log\alpha_i\notin\spn_\mathbb{Q}\{\beta_j\}$. The previous theorem verifies this when all the $\alpha_i,\beta_j$ are algebraic; the next theorem establishes a close variant assuming algebraicty of only the $\alpha_i$, with no assumptions on the nature of the $\beta_j$.

\begin{thm}\label{thm:nonhom} If the scaling constants $\alpha_i$ are algebraic and $\rank_\mathbb{Q}\{\log \alpha_i\}_{i\in I}>|J|$, then $\mathcal{E}=\emptyset$.
\end{thm}

For example, a self-similar set that has scaling constants $\frac{1}{3},\frac{1}{4},\frac{1}{5}$ does not affinely embed in any strongly separated self-similar set  generated by two maps, and the intersection with such a set leads to some dimension reduction. We are not aware of any other proof for this elementary statement.

Our methods also give strong bounds on the possible size of the set of exceptional parameters for the conjecture:

\begin{thm}\label{thm:generic}
Fix $I,J$. Then Conjecture \ref{conj:main} holds for Lebesgue-almost all choices of $(\alpha_{i})_{i\in I},(\beta_j)_{j\in J}$ (and any choice of the translations in the maps). Moreover,
\begin{enumerate}
\item For  $\{\beta_j\}_{j\in J}$ fixed, the Conjecture holds for all but a  set of $\{\alpha_i\}_{i\in I}$ of zero packing dimension in $\mathbb{R}^{|I|}$.
\item For $\{\alpha_i\}_{i\in I}$ fixed, the Conjecture holds for all but a subset of at most $|J|-1$ packing dimension of  $\{\beta_j\}_{j\in J}$ in  $\mathbb{R}^{|J|}$. If in addition the $\alpha_i$ are algebraic, then the exceptional set of $\{\beta_j\}_{j\in J}$ has packing dimension at most $k-d$, where $d=\rank_\mathbb{Q}\{\log\alpha_i\}_{i\in I}$.
\end{enumerate}
\end{thm}


All three theorems above are derived from a general result which confirms the Conjecture under a Diophantine-type condition. To state it we shall assume that $X$ is homogeneous, and denote by $\alpha=\alpha_i$ the common scaling constant. This is not a serious restriction since it is well known that the full conjecture follows from this case,  see e.g. \cite[proof of Corollary 7.6]{Shmerkin2019} or Section \ref{sec:remaining_proofs} below. Let
\[
   \Lambda=\Lambda(X,Y):=\{\frac{\log\beta_{j}}{\log\alpha}\mid j\in J\}.
\]
With this setup, the conjecture asserts that $\mathcal{E}=\emptyset$ whenever   $1\notin\spn_\mathbb{Q}\Lambda$. 

Define a $\Lambda$-multi-rotation to be any sequence $(\theta_{i})\subseteq\mathbb{R}/\mathbb{Z}$
such that $\theta_{i+1}-\theta_{i}\in\Lambda$ for all $i$. Also, let $\overline{\dim}_B$ denote the upper box dimension of a set, and let
\[
\delta(\Lambda)=\inf\left\{ \overline{\dim}_{B}\{\theta_{i}\}\mid(\theta_{i})_{i=1}^{\infty}\text{ is\ a \ensuremath{\Lambda}-multi-rotation}\right\}. 
\]
 Feng and Xiong \cite{FengXiong2018} have shown that $\delta(\Lambda)>1/(\rank_\mathbb{Q}\{\log\beta_j\}+1)$, and in particular $\delta(\Lambda)>0$. 

Our main technical Theorem requires  $\Lambda$-multi-rotations that are robust, in a certain sense, under deletion of elements, as long as a positive-density set of elements remains. Denote the density of $U\subseteq\mathbb{N}$ by $$
d(U)=\lim\frac{1}{N}|U\cap\{1,\ldots,N\}|
$$
assuming the limit exists, and define the upper and lower densities, $\overline{D}(U),\underline{d}(E)$, using limsup and liminf respectively. We then define the following ``robustness'' properties of $\Lambda$.

\begin{description}
\item[Condition\ (\uR)]:  If $(\theta_{i})$ is a $\Lambda$-multi-rotation  and  $U\subseteq\mathbb{N}$ satisfies $\underline{d}(I)>0$, then   $\overline{\dim}_{B}\{\theta_{i}\mid i\in U\}>0 $.

\item[Condition\ (\oR)]:  Defined similarly using upper, instead of lower, density. 
\end{description}

Clearly Condition (\oR) implies Condition (\uR).

\begin{thm}\label{thm:main}
Let $X,Y$ be strongly separated self-similar sets with $X$   homogeneous and $\Lambda$ as above. If $\Lambda$ satisfies Condition (\uR) (or the stronger Condition (\oR)), then Conjecture \ref{conj:main} holds.
\end{thm}

Theorem \ref{thm:main} implies all of the previous stated results as special cases. 
To derive  Theorem \ref{thm:algebraic}, we first reduce to the case of homogeneous $X$, and then  use $\log\alpha_i\notin\spn_\mathbb{Q}\{\beta_j\}$ and algebraicity of $\alpha_i,\beta_j$ to derive Condition (\uR) from Baker's bounds on the distance between logarithms of algebraic numbers \cite{Baker1972}. We remark that one can similarly
deduce other cases of the Conjecture; for instance, a theorem of W. Schmidt  implies
Condition (\uR) when $\log\alpha_{i},\log\beta_{i}$
are algebraic.  The proof of Theorems  \ref{thm:nonhom} and \ref{thm:generic} is based on the observation that  Condition (\uR) is very much the generic situation. For example, for Theorem \ref{thm:nonhom}, we show that if there are sufficiently many $\alpha_i$, then for any choice of $\beta_j$'s there exists some $\alpha_i$ for which the corresponding set $\Lambda$ satisfies Condition (\uR). 

Our final result is of a different nature, but is also stated in terms of the quantity $\delta(\Lambda)$:

\begin{thm}
\label{thm:main-half} Conjecture \ref{conj:main} holds if   $\dim Y<\delta(\Lambda)$, and in particular it holds whenever $\delta(\Lambda)=1$.
\end{thm}
This improves earlier results in which the condition took the form $\dim Y<c\delta(\Lambda)$ for a constant $c\leq1/2$, depending on the context. See e.g. \cite{Baker2021, FengHuangRui2014, FengXiong2018}. For  example, Feng and Xiong show in \cite{FengXiong2018}, that when $|J|=2$, the hypothesis of the conjecture implies that $\delta(\Lambda)\geq 1/2$, and they deduced that the conjecture holds whenever $\dim Y<\frac{1}{4}$. By Theorem \ref{thm:main-half}, we can improve this to $\dim Y<\frac{1}{2}$.

The dimension properties of multi-rotations expressed by the last two theorems hint at the possibility that stronger results on multi-rotations might resolve the conjecture in full: It would follow if, 
assuming $1 \notin\spn_\mathbb{Q}\Lambda$, every $\Lambda$-multi-rotation
had full box dimension (as has been suggested by Yu \cite{Yu2019}),
or if they always had positive box dimension when restricted to positive-density
subsequences (indeed,  Feng-Xiong's lower bounds on the box dimension of complete multi-orbits \cite{FengXiong2018} imply this when a zero-density set is deleted). Unfortunately, it turns out that there are counterexamples to each of these statements; see \cite{Hochman2024}. It is still conceivable that every $\Lambda$ must satisfy one of these two properties, which would imply the conjecture. But we have no evidence that such a statement should hold.

Finally let us say a few words about the proof of Theorem \ref{thm:main}, our main technical result. We begin with the by-now standard observation that if $\mathcal{E}$ contains one element, then the symmetries of $X,Y$ allow us to produce an infinite sequence of embeddings whose scaling constants form a $\Lambda$-multi-rotation (see e.g. \cite{FengXiong2018}). 
There are two new ingredients in the proof of Theorem \ref{thm:main}. The main one is a non-conformal renormalization argument in the space $\mathcal{E}$, which ``magnifies'' small parts of $\mathcal{E}$ to ``macroscopic size'' without leaving $\mathcal{E}$. This allows us to bring the machinery of CP-distributions into play.  At the same time, the non-conformality of the renormalization procedure means that asymptotically it approximates a linear projection. This opens the door for us to apply Marstrand-type projection theorems, and, specifically, Orponen's
projection theorem for Assouad dimension \cite{Orponen2021}. This, in combination with Condition (\oR), allow us to ``amplify'' the size of $\mathcal{E}$ until its
dimension reaches $1$, and also  shows that $\min \lbrace1, \dim_{B}\mathcal{E}\rbrace \leq\dim Y$.
These two statements are contradictory, showing that $\mathcal{E}=\emptyset$. We note that in the proof of Theorem \ref{thm:main-half}, the application of Orponen's theorem is more direct and constitutes the main innovation.

The paper is organized as follows. In the next section we review our
assumptions and notation and present some background results, including
on Assouad dimension and CP-distributions. In Section \ref{sec:Derivation-of-Theorem-half}
we prove Theorem \ref{thm:main-half}. In Section \ref{sec:Renormalization}
we present and analyze the renormalization procedure on $\mathcal{E}$,
followed in Section \ref{sec:Proof-of-main-result} by the proof of
Theorem \ref{thm:main}. Finally in Section \ref{sec:remaining_proofs} we provide the remaining proofs.

\section{\label{sec:Notation,-assumptions-and-backgroun} Preliminaries}

\subsection{\label{SSSs-and-assumptions}Self-similar sets, and simplifying assumptions}

Let $\Phi=\{\varphi_{i}\}_{i\in I}$ and $\Psi=\{\psi_{j}\}_{j}$
be finite families of affine non-singular contractions of $\mathbb{R}$, given by $\varphi_{i}(x)=\alpha_{i}x+\sigma_{i}$
and $\psi_{j}(x)=\beta_{j}x+\tau_{j}$. Let $X=\bigcup_{i\in I}\varphi_{i}(X)$ and $Y=\bigcup_{j\in J}\psi_{j}(Y)$ be the corresponding self-similar sets. We assume that they are strongly separated and non-trivial, i.e. neither the $\varphi_i$ or the $\psi_j$ have a common fixed point. 

For $\ii=i_{1}\ldots i_{n}\in I^{n}$ and $\jj=j_{1}\ldots j_{n}\in J^{n}$
we denote the compositions of $\varphi_{i_{1}},\ldots,\varphi_{i_{n}}$
and $\psi_{j_{1}},\ldots,\psi_{j_{n}}$, and their coefficients, by
\begin{align*}
\varphi_{\ii}(x) & =\varphi_{i_{1}}\circ\ldots\circ\varphi_{j_{n}}(x)=\alpha_{\ii}x+\sigma_{\ii},\\
\psi_{\jj}(x) & =\psi_{j_{1}}\circ\ldots\circ\psi_{j_{n}}(x)=\beta_{\jj}x+\tau_{\jj}.
\end{align*}
When $X$ is homogeneous with contraction $\alpha$,  we have $\alpha_{\ii}=\alpha^{n}$ when $\ii\in I^{n}$. 

The strong separation property of $Y$ implies  the following standard Lemma.
\begin{lem}
\label{lem:engulfing-cylinders} There exists $\rho>1$ such that for
every set $Z\subseteq Y$, there exists $\jj\in J^{*}$ (which might
be the empty word) such that 
$$Z\subseteq\psi_{\jj}(Y), \text{ and }\diam\psi_{\jj}(Y)<\rho\diam Z.$$
\end{lem}
Throughout this section and Sections \ref{sec:Derivation-of-Theorem-half}, \ref{sec:Renormalization} and \ref{sec:Proof-of-main-result}, we make the following simplifying assumptions,
without loss of generality.
\begin{enumerate}

\item All $\alpha_{i}$ have a common value $\alpha$. 
\item $|\alpha|<\frac{1}{2}$.

This can be achieved by replacing $\{\varphi_{i}\}_{i\in I}$ by $\{\varphi_{\ii}\}_{\ii\in I^{k}}$
for a sufficiently large $k$, noting that all other properties of the maps are preserved. 

\item $0\in X$ and $\diam X=\diam Y=1$. 

This can be achieved by means of a change of coordinates on $\mathbb{R}$.

It follows that $\varphi_{i}(0)\in X$ for all $i\in I$, and consequently also
$\sigma_{\ii}=\varphi_{\ii}(0)\in X$ for all $\ii\in I^{*}$.

Also, if $f \in \mathcal{E}$ then $\left\Vert f\right\Vert \leq1$. 

\item All the maps $\varphi_i,\psi_j$ are orientation preserving.

This assumption is only in order to avoid keeping track of the signs. The same proofs work in the general case if the signs are re-inserted.

Note that this implies $\alpha_i=\Vert\varphi_i\Vert$ and $\beta_j=\Vert \phi_j\Vert$.

\end{enumerate}

\subsection{\label{subsec:Assouad-dimension and Orponen's Theorem}Assouad dimension}
Let $E\subseteq\mathbb{R}^{d}$ be a bounded set. We will work with its
 Hausdorff dimension $\dim E$ and its upper box dimension
$\overline{\dim}_{B}E$.  We will also make use of its Assouad dimension $\dim_{A}E$, which is defined as follows.
Let $\cov(E,\delta)$ denote the minimal number of $\delta$-balls
needed to cover $E$. The Assouad dimension of $E\subseteq\mathbb{R}^{d}$
is
\[
\dim_{A}E=\limsup_{r\rightarrow0}\sup_{R>0}\sup_{x\in E}\frac{\log\cov(B_{R}(x)\cap E,R\cdot r)}{\log(1/r)}
\]
In general, $\overline{\dim}_{B}E\leq\dim_{A}E$, and the inequality
can be strict. However, for strongly separated self-similar sets it is well known that $\dim E=\dim_B E=\dim_A E$.

The following strong Marstrand-type theorem of
Orponen \cite{Orponen2021} for Assouad dimension will play an important role in our analysis. For $\sigma\in\mathbb{R}$ let $\pi_{\sigma}:\mathbb{R}^{2}\rightarrow\mathbb{R}$
be the linear map $\pi_{\sigma}(u,v)=\sigma u+v$.
\begin{thm} \label{Thm Orponen}
Let $E\subseteq\mathbb{R}^{2}$. Then 
\[
\dim \{\sigma\in\mathbb{R}\mid\dim_{A}\pi_{\sigma}E<\min\{1,\dim_{A}E\}\}=0.
\]
\end{thm}

\subsection{\label{subsec:Magnification-and-CP-chains}Magnification and CP-distributions}

For $0\leq n\in\mathbb{Z}$, the level-$n$ dyadic partition of $\mathbb{R}$
is the partition
$$\mathcal{D}_{n}=\{[\frac{k}{2^{n}},\frac{k+1}{2^{n}})\mid k\in\mathbb{Z}\}.$$
The level-$n$ dyadic partition of $\mathbb{R}^{d}$ is the product
partition $\mathcal{D}_{n}^{d}=\{I_{1}\times\ldots\times I_{d}\mid I_{i}\in\mathcal{D}_{n}\}$.
We suppress the dimension and abbreviate $\mathcal{D}_{n}$ for all
$d$; the underlying dimension will be clear from context. For $x\in\mathbb{R}^{d}$ we write $\mathcal{D}_{n}(x)$ for the unique
element of $\mathcal{D}_{n}$ containing $x$.

The notion of a CP--distribution was introduced by Furstenberg in order to
study the structure of sets and measures in terms of what one sees
when one ``zooms in'' on them along ever smaller dyadic cells \cite{Furstenberg1970,Furstenberg2008}.
We summarize the key definitions and results that we shall use from
this theory, specializing to the case $d=2$, which is most relevant to us.

Let $D\in\mathcal{D}_{n}$ be a dyadic cell. We can magnify  $D$
homothetically to a set of ``macroscopic size'' using the orientation-preserving
homothety mapping $D$ onto $[0,1)^{2}$. We denote this map by 
\[
H_{D}:\mathbb{R}^{2}\rightarrow\mathbb{R}^{2}.
\]
For a probability measure $\mu$ on $\mathbb{R}^{d}$ and $D\in\mathcal{D}_{n}$
with $\mu(D)>0$, denote the conditional measure of $\mu$ on $D$
by 
\[
\mu_{D}=\frac{1}{\mu(D)}\mu|_{D},
\]
and define the magnified version of $\mu_{D}$ by 
\[
\mu^{D}=H_{D}\mu_{D}.
\]

\begin{defn}
\label{def:A-CP-distribution-is}A CP-distribution on $\mathbb{R}^{d}$
is a probability measure $P$ on the space $\mathcal{P}([0,1)^{d})$
of probability measures on $[0,1)^{d}$, which, for every $n\in\mathbb{N}$
(equivalently, for $n=1$) is stationary for the Markov transition
\[
\mu\mapsto\mu^{D}\qquad\text{with $D\in \mathcal{D}_{n}$ chosen with probability \ensuremath{\mu(D)}}.
\]
We say that  $P$ is ergodic if the corresponding stationary Markov chain is ergodic. 
\end{defn}
Recall that a measure $\mu \in \mathcal{P}(\mathbb{R}^d)$ is exact dimensional with dimension $\delta$ if
\[
\mu(B_{r}(x))=r^{\delta+o_{\mu,x}(1)}\qquad\text{for \ensuremath{\mu}-a.e. \ensuremath{x}}.
\]
If this holds then $\dim\supp\mu\geq\delta$. The following Theorem is proved in various forms in the works of Furstenberg \cite{Furstenberg2008}, Hochman \cite{hochman2010dynamics}, and Hochman-Shmerkin \cite{hochman2009local}.
\begin{thm}
\label{thm:dim-of-CP}If $P$ is an ergodic CP-distribution on $\mathbb{R}^{d}$,
then there exists a $\delta>0$ such that $P$-a.e. $\mu$ is exact
dimensional with dimension $\delta$. Moreover, if $\delta=d$ then
$P$-a.e. $\mu$ is equal to Lebesgue measure on $[0,1)^{d}$.
\end{thm}
For $P,\delta$ as in the theorem, we then say that $P$ is $\delta$-dimensional.

Next a sequence of dyadic cells $D_{1}\supseteq D_{2}\supseteq\ldots$
is called consecutive if there is some $n$ such that $D_{i}$ is from level
$n+i$. The next Theorem is essentially due to Furstenberg \cite{Furstenberg2008}, although the connection with Assouad dimension was first stated in  \cite{kaenmaki2015weak}. Here and for the remainder of the paper, we assume that a metric has been fixed on the space of measures that is compatible with the weak-* topology.
\begin{thm}
\label{thm:realizing-Assouad-dim-with-CP} Let $E\subseteq\mathbb{R}^{d}$
be a bounded set with $\delta:=\dim_{A}E$. Then there exists a $\delta$-dimensional
ergodic CP-distribution $P$ with the following property. For every
$\varepsilon>0$, there exist arbitrarily large $N\in\mathbb{N}$,
consecutive dyadic cells $D_{1}\supseteq D_{2}\supseteq\ldots\supseteq D_{N}$,
and a probability measure $\nu$ on $E$ with $\nu(D_{N})>0$, such
that, when $k$ is chosen uniformly from $\{1,\ldots,N\}$, the distribution
of the random measure $\nu^{D_{k}}$ is $\varepsilon$-close to $P$.
\end{thm}
Finally, the next Theorem is a standard application of the ergodic Theorem.  
\begin{thm}
\label{thm:CP-equidistribution}Let $P$ be an ergodic CP-distribution
and $F\in L^{1}(P)$. Then for $P$-a.e. $\mu$ and $\mu$-a.e. $x$,
 
\[
\frac{1}{N}\sum_{n=1}^{N}F(\mu^{\mathcal{D}_{n}(x)})\rightarrow\int F(\nu)dP(\nu)\qquad\text{as \ensuremath{N\rightarrow\infty}}.
\]
\end{thm}

\section{\label{sec:Derivation-of-Theorem-half}Proof of Theorems \ref{thm:main-half}}

Let $X,Y \subseteq \mathbb{R}$ be non-trivial self-similar sets with the strong separation condition, and satisfying the conditions stated in Section \ref{SSSs-and-assumptions}. Recall the definition of $\mathcal{E}$,  $\Lambda$, and $\delta(\Lambda)$ from the introduction.

Suppose that $\mathcal{E}\neq\emptyset$. We will show that $\dim Y\geq \delta(\Lambda)$. We need two lemmas:
\begin{lem} \label{lem: 1}
If $\mathcal{E}\neq\emptyset$ then there exists a compact subset $\mathcal{F}\subseteq\mathcal{E}$ such that 
$$\overline{\dim}_B \lbrace a: \exists b\in \mathbb{R} \text{ with } (a,b)\in \mathcal{F} \rbrace \geq \delta(\Lambda).$$
\end{lem}
This is standard. It follows  from e.g.
\cite{FengXiong2018}, or can be recovered from Proposition \ref{prop:renomalization-at-point-gives-alpha-beta-set} below.

\begin{lem} \label{lem: 2}
The strongly separated self-similar set $Y$ satisifies $\dim Y = \dim_A Y$.
\end{lem}
This is due to Furstenberg; it follows from the
fact that anything happens in $Y\cap B_{R}(x)$ can be re-scaled to
scale $\Theta(1)$ using Lemma \ref{lem:engulfing-cylinders} and
the maps $\psi_{\jj}^{-1}$. This means that the local covering numbers
used to define $\dim_{A}Y$ occur, with negligible error, also as
covering numbers of $Y$ itself. For a full proof, see \cite{bishop2013fractal}.

Let us now prove Theorem \ref{thm:main-half}. Having assumed that $\mathcal{E}\neq \emptyset$, let $\mathcal{F\subseteq\mathcal{E}}$ be
as in Lemma \ref{lem: 1}.

Fix $\sigma\in X$. Let $\pi_{\sigma}:\mathcal{E}\rightarrow Y$ denote the map taking $f(x)=ax+b$ to $\pi_{\sigma}(f)=f(\sigma)=a\sigma+b$. Note that the image of $\pi_\sigma$ is contained in $Y$  because $f(X)\subseteq Y$.

As a function of $(a,b)\in\mathbb{R}^{2}$ the map $\pi_{\sigma}$
is a linear functional. So, applying Orponen's Theorem \ref{Thm Orponen}, 
$$\dim_{A}\pi_{\sigma}(\mathcal{F})\ge \min\{1,\dim_{A}\mathcal{F}\}$$
for all but at most a $0$ Hausdorff dimensional set of $\sigma \in X$.  Since $X$ is non-trivial  $\dim X>0$,  hence there exists some $\sigma\in X$ with
\[
\dim_{A}\pi_{\sigma}(\mathcal{F})\ge  \min\{1,\dim_{A}\mathcal{F}\} \geq\min\{1,\overline{\dim}_{B}\mathcal{F}\}\geq\delta(\Lambda).
\]
Since $Y\supseteq\pi_{\sigma}(\mathcal{F})$, we conclude that $\dim_{A}Y\geq\delta(\Lambda)$. Applying Lemma \ref{lem: 2} we obtain the desired conclusion $\dim Y\geq\delta(\Lambda)$.

\section{\label{sec:Renormalization} Renormalization  on the space $\mathcal{E}$}

\subsection{The space of affine maps} Let $\mathcal{A}$ be the space of   non-singular affine maps of $\mathbb{R}$; that is,
\[
\mathcal{A}=\{f(x)=a\cdot x+b:\,a\neq 0\}.
\]
We introduce coordinates, identifying $f(x)=a\cdot x+b$ with $(a,b)\in\mathbb{R}^{2}$. Thus,
 $\mathcal{A}$ is identified with $\mathbb{R}^{2}\setminus( \lbrace 0 \rbrace \times\mathbb{R})$. Topologically, $\mathcal{A}$ consists of two copies of $\mathbb{R}^{2}$, each copy corresponding to one orientation. 

This parameterization has several shortcomings, primarily that the
metric induced on $\mathcal{A}$ from $\mathbb{R}^{2}$ is neither
complete nor proper (bounded subsets of $\mathcal{A}$ need not be
pre-compact). However, we shall work exclusively with subsets
of $\mathcal{A}$ that are either compact or pre-compact. In this case
the metric induced from $\mathbb{R}^{2}$ is entirely appropriate.


\subsection{\label{subsec:Dyadic-partition-of-A}Dyadic cells in $\mathcal{A}$}

The identification of $\mathcal{A}$ with a subset of $\mathbb{R}^{2}$ induces partitions of $\mathcal{A}$ from the dyadic partitions
of $\mathbb{R}^{2}$. We denote these also by $\mathcal{D}_{n}$,
the underlying space should be clear from the context.

It is important to note that the cells of these partitions need not
be pre-compact in $\mathcal{A}$. Under our identification, a cell $D$ is pre-compact if and only if its closure is disjoint from the $x$-axis.  We shall restrict the discussion to
pre-compact cells, and use this term only in reference to subsets of $\mathcal{A}$. Note that if $\mathcal{F}\subseteq\mathcal{E}$
is compact (in particular when $\mathcal{F}=\{f\}$), then there exists
an $n_{0}$ so that for $n>n_{0}$ every cell $D\in\mathcal{D}_{n}$
that intersects $\mathcal{F}$ is pre-compact. 

Let $D\in\mathcal{D}_{n}$ and set
\[
\chi_{D}:=\sup_{f\in\mathcal{E}\cap D}\left\Vert f\right\Vert. 
\]
Note that $0<\chi_{D}\leq1$ since $\mathcal{E}$ consists of non-singular contractions. 
\begin{lem}
For any $f,g\in\mathcal{E}\cap D$, $D\in D_n$, 
we have for all $x,y\in X$
\begin{equation} \label{eq:continuity-of-action}
|f(x)-g(y)|\leq  \chi_{D}|x-y|+2\cdot2^{-n}.
\end{equation}
\end{lem}
\begin{proof}
For $f,g\in\mathcal{E}\cap D$, writing $f(x)=ax+b$ and $g(x)=cx+d$,  as  $|a-c|$ and $|b-d|$ are both $<2^{-n}$, we have
\begin{align}
|f(x)-g(y)| & \leq|a(x-y)+(a-c)y+(b-d)|\nonumber \\
 & \leq|a||x-y|+|a-c||y|+|b-d|\nonumber \\
 & \leq\chi_{D}|x-y|+2\cdot2^{-n}
\end{align}
where we have used $|a|\leq\chi_{D}$ and also the fact that $0\in X$
and $\diam X=1$ to conclude $|y|\leq1$.
\end{proof}

\subsection{\label{subsec:Renormalization-in-E}The renormalization maps $M_{D,\ii}$}

Using the identification of $\mathcal{A}$ with a subset of $\mathbb{R}^{2}$
and the associated dyadic partitions, we can define the magnifying
maps $H_{D}$ on $\mathcal{A}$ as in Section \ref{subsec:Magnification-and-CP-chains}. However, this operation does not
generally preserve $\mathcal{E}$.

We introduce now a different re-scaling procedure which we call renormalization.
It is defined only on $\mathcal{E}$ (rather than all of $\mathcal{A})$,
but it preserves $\mathcal{E}$. Renormalization is non-conformal,
which we shall see later is both an advantage and a disadvantage.

Consider a cell $D\in\mathcal{D}_{n}$ such that $D\cap\mathcal{E}\neq\emptyset$.  Suppose furthermore  that $\overline{D}$ is compact, contained entirely in the
space of properly contracting maps so that $\chi_{D}<1$, and 
that $2^{-n}/\chi_{D}\leq1$. Note that for every compact $\mathcal{F}\subseteq\mathcal{E}$,
for every large enough $n$ this holds for every $D\in\mathcal{D}_{n}$
intersecting $\mathcal{F}$.

Define a renormalization map $\mathcal{E}\cap D\rightarrow\mathcal{E}$
as follows. Recall that $\alpha$ is the common scaling constant of the maps $\varphi_i$.
\begin{enumerate}
\item [(R1)] Choose $k\geq0$ to be the unique integer such that 
\begin{equation}
\alpha\frac{2^{-n}}{\chi_{D}}\leq\alpha^{k}<\frac{2^{-n}}{\chi_{D}}\label{eq:choice-of-k}
\end{equation}
When we want to indicate the dependence on $n$ we write $k=k(n)$.
\item [(R2)] Choose $\ii\in I^{k}$ arbitrarily. 

Observe that, because of the assumptions $\diam X=1$ and $\alpha_{i}=\alpha$,
\[
\alpha\frac{2^{-n}}{\chi_{D}}\leq\diam\varphi_{\ii}(X)=\alpha^k<\frac{2^{-n}}{\chi_{D}}
\]
Hence, by (\ref{eq:continuity-of-action})
and the definition of $\chi_{D}$, the set 
\[
Z_{D,\ii}:=\bigcup_{f\in\mathcal{E}\cap D}f(\varphi_{\ii}(X))\subseteq X
\]
has diameter bounded by 
$$\chi_{D}\diam\varphi_{\ii}X+2\cdot2^{-n}\leq3\cdot2^{-n}.$$
\item [(R3)] Let $\rho$ be the constant from Lemma \ref{lem:engulfing-cylinders}.
Choose $\jj\in J^{*}$ of minimal length such that $\left\Vert \psi _{\jj}\right\Vert<\rho\cdot3\cdot2^{-n}$
and $Z_{D,\ii}\subseteq\psi_{\jj}(Y)$. Note for future use that
\begin{equation}
\left\Vert \psi_{\jj}\right\Vert =\Theta(2^{-n}).\label{eq:psi-j-bound}
\end{equation}
By our choice of $\jj$ we have $Z_{D,\ii}\subseteq\psi_{\jj}(Y)$.
So, for every $f\in\mathcal{E}\cap D$, 
\[
\psi_{\ii}^{-1}f\varphi_{\ii}(X)\subset \psi_{\jj}^{-1}Z_{D,\ii}\subseteq Y.
\]
Therefore, $\psi_{\jj}^{-1}f\varphi_{\ii}\in\mathcal{E}$.
\end{enumerate}
We define the renormalization map $M_{D,\ii}:\mathcal{E}\cap D\rightarrow\mathcal{E}$
by
\[
M_{D,\ii}f=\psi_{\jj}^{-1}f\varphi_{\ii} \;\;\;\;\textrm{ for $f\in\mathcal{E}\cap D$}.
\]

Since $X,Y$ are compact, the translation parts of the maps in $\mathcal{E}$
are also automatically bounded. Since the image of $M_{D,\ii}$
is contained in $\mathcal{E}$, it consists of contractions: 
$$\left\Vert M_{D,\ii}(f)\right\Vert \leq 1.$$
Furthermore, for $f\in\mathcal{E}\cap D$, we have $\left\Vert f\right\Vert \geq\chi_{D}/\sqrt{e}$,
$\left\Vert \varphi_{\ii}\right\Vert \geq\alpha\frac{2^{-n}}{\chi_{D}}$
and $\left\Vert \psi_{\jj}\right\Vert <\rho\cdot3\cdot2^{-n}$. This gives a lower bound on the scaling constant,
\[
\left\Vert M_{D,\ii}(f)\right\Vert =\frac{\left\Vert f\right\Vert \left\Vert \varphi_{\ii}\right\Vert }{\left\Vert \psi_{\jj}\right\Vert }>\frac{(\chi_{D}/\sqrt{e})\cdot\alpha2^{-n}/\chi_{D}}{\rho\cdot3\cdot2^{-n}}=\frac{\alpha}{3\rho\sqrt{e}}.
\]
Thus, $M_{D,\ii}$ maps
$D$ into the fixed compact set 
\begin{equation} \label{eq: E zero}
\mathcal{E}_{0}=\{g\in\mathcal{E}\mid\left\Vert g\right\Vert \geq\frac{\alpha}{3\rho\sqrt{e}}\}.
\end{equation}

The renormalization map $M_{D,\ii}$ depends on the choice $\ii\in I^{n}$.
So, potentially there are $|I|^{n}$ different renormalizations for
the same $D\in\mathcal{D}_{n}$. This freedom to choose $\ii$ will
be significant later. On the other hand, an important case occurs when we choose $\ii$ to be a prefix of a fixed infinite word in $I^\mathcal{N}$. This is closely
related to the standard ``amplification of embeddings'' method that leads to Lemma \ref{lem: 1}.

We say that a subsequence $\theta_{n_{i}}$ of a sequence $\theta_{n}$ is syndetic 
if $n_{i+1}-n_{i}$ is bounded. Recall the definition of $\Lambda$ as in Section \ref{sec:Introduction}.
\begin{prop}
\label{prop:renomalization-at-point-gives-alpha-beta-set}Let $f\in\mathcal{E}$
and let $D_{n}=\mathcal{D}_{n}(f)$ denote the dyadic cell of level
$n$ containing $f$. Let $\ii\in I^{\mathbb{N}}$ be any fixed sequence.
Let $k(n)$ denote the integer associated to $D_{n}$ in step (R1),
let $\ii_{n}=\ii|_{[1,k(n)]}$, and let $\jj_{n}\in J^{*}$ denote
the sequence associated to $D_{n},\ii_{n}$ in step (R3). 

Then $\jj_{n+1}$ extends $\jj_{n}$, and the sequence
\[
\theta_{n}=\frac{1}{\log\alpha}\log\left\Vert M_{D_{n},\ii_{n}}(f)\right\Vert 
\]
contains a syndetic subsequence that is itself a syndetic subsequence
of a $\Lambda$-multi-rotation.
\end{prop}
\begin{proof}
Clearly $k(n)$ is non-decreasing
and if $k(n+1)=k(n)$ then the definition dictates that $\ii_{n+1}=\ii_{n}$,  $\jj_{n+1}=\jj_{n}$ and $\theta_{n+1}=\theta_n$. But this can occur only for
a bounded number of consecutive $n$, because there is a uniform bound
$K$ on the number of integers $n$ satisfying $\alpha2^{-n}\leq\alpha^{k}<2^{-n}$, or equivalently, $\alpha^{k}<2^{-n}\leq\alpha^{k-1}$. 

The subsequence of $(\theta_{n})$ that we are interested
in is the subsequence along the times when $\theta_{n+1}-\theta_{n}>0$.
The previous paragraph shows that this is a syndetic subsequence.

Also, since we are assuming $\alpha<1/2$, it follows that if $k(n+1)>k(n)$,
then $k(n+1)=k(n)+1$.

Consider an $n$ such that $k(n+1)=k(n)+1$. Then $\ii_{n+1}=\ii_{n}u$
for some symbol $u\in I$. So $\varphi_{\ii_{n+1}}=\varphi_{\ii_{n}}\circ\varphi_{u}$.
Since $\varphi_{u}(X)\subseteq X$ we have $\varphi_{\ii_{n+1}}(X)\subseteq\varphi_{\ii_{n}}(X)$.
Since also $D_{n+1}\subseteq D_{n}$, we get
\[
Z_{D,\ii_{n+1}}=\bigcup_{g\in \mathcal{E}\cap D_{n+1}}g\varphi_{\ii_{n+1}}(X)\subseteq\bigcup_{g\in \mathcal{E}\cap D_{n}}g\varphi_{\ii_{n}}(X)=Z_{D_{n},\ii_{n}}\subseteq\psi_{\jj_{n}}(Y).
\]
Thus $Z_{D,\ii_{n+1}}\subseteq\psi_{\jj_{n}}(X)\cap\psi_{\jj_{n+1}}(X)$. By the strong separation of $Y$, one of the words $\jj_{n},\jj_{n+1}$ must
be a prefix of the other. Since $\left\Vert \psi_{\jj_{n+1}}\right\Vert <\rho\cdot3\cdot2^{-n}$,
if $\jj_{n+1}$ were a proper prefix of $\jj_{n}$ we would be in
contradiction of the minimality condition in step (R3) of the definition
of $\jj_{n}$. Thus, $\jj_{n+1}$ extends $\jj_{n}$.

Next, if $\jj_{n+1}$ is a proper extension of $\jj_{n}$ then $\jj_{n+1}=\jj_n\bf{v}$
for some ${\bf{v}}\in J^{*}$ of positive length $\ell$. Then $\ell$
is the minimal length needed to achieve $\beta_{\jj_{n+1}}=\left\Vert \psi_{\jj_{n+1}}\right\Vert <\rho\cdot3\cdot2^{-n-1}$.
Writing $\beta_{\max}=\max\{|\beta_{j}|\mid j\in J\}$, we have 
\[
\beta_{\jj_{n+1}}=\beta_{\bf{v}}\cdot\beta_{\jj_{n}}\leq\beta_{\max}^{\ell}\cdot\beta_{\jj_{n}}\leq\beta_{\max}^{\ell}\cdot\rho\cdot3\cdot2^{-n}.
\]
This gives a uniform upper bound $L$ for $\ell$. Now, 
\[
\left\Vert M_{D_{n+1},\ii_{n+1}}\right\Vert =\left\Vert \psi_{\jj_{n}\bf{v}}^{-1}f\varphi_{\ii_{n}u}\right\Vert =|\beta_{\bf{v}}^{-1}\alpha|\cdot\left\Vert \psi_{\jj_{n}}^{-1}f\varphi_{\ii_{n}}\right\Vert =|\beta_{\bf{v}}^{-1}\alpha|\cdot\left\Vert M_{D_{n1},\ii_{n}}\right\Vert .
\]
Taking logarithms and dividing by $\log\alpha$, we see that 
\[
\theta_{n+1}-\theta_{n}\in\frac{\log\beta_{\bf{v}}^{-1}}{\log\alpha}+1
\]
so modulo one, writing ${\bf{v}}=v_{1}\ldots v_{\ell}$, we have $\theta_{n+1}-\theta_{n}\equiv\sum_{i=1}^{\ell}\frac{\log\beta_{v_{i}}}{\log\alpha}$.
Since $\ell\leq L$ this is a syndetic subsequence of $\Lambda$-multi-rotation. 
\end{proof}
We remark that, unlike the magnification used in the theory of CP-distributions,
iterating the renormalization map does not generally yield a map that
is itself a renormalization. That is, given $D\in\mathcal{D}_{n}$,
$D'\in\mathcal{D}_{k}$ and given $\ii\in I^{n}$, $\ii'\in I^{m}$, 
the map $M_{D',\ii'}\circ M_{D,\ii}$ in general will not equal $M_{D'',\ii''}$
for any dyadic cell $D''\in\mathcal{D}_{m+n}$ and $\ii''\in I^{n+m}$.
For instance, $M_{D,\ii}^{-1}(D')$ need not be a dyadic cell. These
maps do enjoy an approximate closure under composition, but we will
not use this fact.

\subsection{\label{subsec:Computation-of-renormalization}Computation of the
renormalization map}

Let $D,\ii,\jj$ be as in the definition of $M_{D,\ii}$, and let
us compute $M_{D,\ii}$ more explicitly. 
\begin{lem}
In our coordinates on $\mathcal{A}$, the map $M_{D,\ii}$ is
the affine transformation of $\mathcal{E}\cap D$ given by, 
for  $f(x)=ax+b$, 
$$ (M_{D,\ii}f)(x) =  \left( \beta_{\jj}^{-1}\alpha_{\ii}\cdot a\right)\cdot x +\beta_{\jj}^{-1}\cdot\left((a\sigma_{\ii}+b)-\tau_{\jj}\right),$$
where $(\alpha_{\ii},\sigma_{\ii})$ and $(\beta_{\jj},\tau_{\jj})$ are the corresponding parameters of $\varphi_{\ii},\psi_{\jj}$.

In particular, in our coordinates on $\mathcal{A}$, the map is affine and is given by \[
  (a,b)\mapsto ( \beta_{\jj}^{-1}\alpha_{\ii}\cdot a \;,\; \beta_{\jj}^{-1}\cdot((a\sigma_{\ii}+b)-\tau_{\jj})).
\]
\end{lem}
We omit the proof, which is a straightforward computation.

The key property we shall use is that, up to a small perturbation of each coordinate
in the image, this map behaves like 
\[
(a,b)\mapsto(a,\pi_{\sigma_{\ii}}\circ H_{D}(a,b)).
\]
Here, $\pi_x$ is the linear functional appearing in Theorem \ref{Thm Orponen}.
\begin{prop}
\label{prop:explicit-renormalization-map}For every compact $\mathcal{F}\subseteq\mathcal{E}$
there exists a compact $\mathcal{H}\subseteq\mathcal{A}$ with the
following property. Let $D\in\mathcal{D}_{n}$ be pre-compact with
$\mathcal{F}\cap D\neq\emptyset$, and let $\ii\in I^{k(n)}$. Then there
exist $h_{1},h_{2}\in\mathcal{H}$ such that for all $(a,b)\in\mathcal{F}\cap D$,
\[
M_{D,\ii}(a,b)=(h_{1}(a)\;,\;h_{2}\circ\pi_{\sigma_{\ii}}\circ H_{D}(a,b))
\]
\end{prop}
\begin{proof}
Let $f\in\mathcal{F}\cap D$ and write $f(x)=ax+b$ and $(M_{D,\ii}f)(x)=a'x+b'$.

For the first coordinate of $M_{D,\ii}(a,b)$, we have seen that $a'=\beta_{\jj}^{-1}\alpha_{\ii}a$. By (\ref{eq:choice-of-k}) and (\ref{eq:psi-j-bound}) we have
$\beta_{\jj}^{-1}\alpha_{\ii}=\Theta(1)$. Thus $a'$ is a multiple
of $a$ by a constant of order $\Theta(1)$. This determines $h_1$.

We turn to the second coordinate. We must compare the affine maps
\begin{align*}
g_{1}:(u,v) & \mapsto b'=\beta_{\jj}^{-1}\cdot\left(\pi_{\sigma_{\ii}}(u,v)-\tau_{\jj}\right)\\
g_{2}:(u,v) & \mapsto\pi_{\sigma_{\ii}}\circ H_{D}(u,v).
\end{align*}
The scaling  constants of $g_{1},g_{2}$ are $\beta_{\jj}^{-1}\sigma_{\ii}$
and $2^{n}\sigma_{\ii}$, respectively. By (\ref{eq:psi-j-bound})
again, these differ by a multiplicative constant of order $\Theta(1)$.
In order to compare the translation parts of $g_{1},g_{2}$, we note
that by what we have already seen,
\begin{align*}
|g_{1}(0,0)-g_{2}(0,0)| & =|(g_{1}(a,b)-\beta_{\ii}^{-1}\sigma_{\ii}a)-(g_{2}(a,b)-2^{n}\sigma_{\ii}a)|\\
 & =|g_{1}(a,b)-g_{2}(a,b)|+\Theta(|a|)\\
 & =|g_{1}(a,b)-g_{2}(a,b)|+\Theta_{\mathcal{F}}(1).
\end{align*}
So it suffices  to show that $|g_{1}(a,b)-g_{2}(a,b)|=O(1)$.
And indeed, $g_{1}(a,b)$ is just the second coordinate of $M_{D,\ii}(a,b)$,
which lies in the compact set $Y$, so $g_{1}(a,b)=O(1)$.
We also have $(a,b)\in D$, so $H_{D}(a,b)\in[0,1)^{d}$, and $\sigma_{\ii}=\varphi_{\ii}(0)\in X$,
so $\sigma_{\ii}=O(1)$. Together these imply 
\[
\pi_{\sigma_{\jj}}H_{D}(a,b)=O(1).
\]
Thus $|g_{1}(a,b)-g_{2}(a,b)|=O_{\mathcal{F}}(1)$ as desired.
\end{proof}

\subsection{Limiting behavior of renormalization }

We now study the behavior of renormalization of measures on increasingly small
dyadic cells in $\mathcal{E}$. Given a measure $\mu$ with $\mu(\mathcal{E}\cap D)>0$ for some dyadic cell $D$, and a suitable $\ii\in I^k$, 
define the renormalization of $\mu$ by 
$$\mu ^{D,\ii}:= M_{D,\ii} \mu.$$
\begin{prop}
\label{prop:limits-of-renormalized-measures} Let  $\mathcal{F}\subseteq\mathcal{E}$ be a compact subset. Then there exists a compact subset $\mathcal{H}\subseteq \mathcal{A}$ such that the following holds. Suppose that for each
$k\in\mathbb{N}$ we are given :
\begin{enumerate}

\item A sequence of integer scales $n_{k}\rightarrow\infty.$
\item A dyadic cell $D_{k}\in\mathcal{D}_{n_{k}}$ with $D_{k}\cap\mathcal{F}\neq\emptyset$
and $D_{k}$ pre-compact.
\item A probability measure $\mu_{k}$ supported on $D_{k}\cap\mathcal{F}$,
such that $\mu=\lim\mu^{D_{k}}$ exists.
\item A sequence $\ii_{k}\in I^{*}$ such that $\sigma=\lim\sigma_{\ii_{k}}$
and $\widetilde{\mu}=\lim\mu_{k}^{D_{k},\ii_{k}}$ exist.
\end{enumerate}
Then there exists $z\in\mathbb{R}$ and  some non-singular  $h\in\mathcal{H}$ with
\[
\widetilde{\mu}=\delta_{z}\times h\pi_{\sigma}\mu,
\]
and $\widetilde{\mu}$ is supported on $\mathcal{E}_0$.
\end{prop}
Note that since all the parameters in the statement lie in compact
sets, convergence of the various sequences can always be achieved
by passing to a subsequence. So the statement really characterizes
accumulation points arising from general sequences $n_{k},D_{k},\ii_{k},\mu_{k}$. 
\begin{proof}
Choose $(p_{q},q_{k})\in\mathcal{F}\cap D_{k}$. 

Let $\mathcal{H}$ be as in Proposition \ref{prop:explicit-renormalization-map}.
Then for all $k$ there are $h_{k,1},h_{k,2}\in\mathcal{H}$ such that 
\begin{align*}
M_{D_{k},\ii_{k}}(u,v) & =(h_{k,1}(u)\,,\,h_{k,2}\circ\pi_{\sigma_{\ii_{k}}}\circ H_{D_{k}}(u,v)).
\end{align*}

Consider the measure
\[
\mu_{k}^{D_{k},\ii_{k}}=M_{D_{k},\ii_{k}}(\mu_{k})
\]
Since $h_{k,1}$ lies in the compact set $\mathcal{H}$ and since
the first marginal of $\mu_{k}$ is supported on a $2^{-n_{k}}$-ball
around $p_{k}$, the first marginal of $\mu^{D_{k},\ii_{k}}$ is supported
on a $\Theta(2^{-n_{k}})$-ball around $z_{k}=h_{k,1}(p_{k})$. Also,
the second marginal of $\mu_{k}^{D_{k},\ii_{k}}$ is supported on
a compact set independent of $k$. Therefore, 
\begin{align*}
\mu_{k}^{D_{k},\ii_{k}} & =\delta_{z_{k}}\times h_{k,2}\pi_{\sigma_{\ii_{k}}}H_{D_{k}}\mu+o(1)
\end{align*}
where the $o(1)$ error should be understood in the weak-{*} sense.
Noting that $H_{D_{k}}\mu_{k}=\mu_{k}^{D_{k}}$, we have
\[
\mu_{k}^{D_{k},\ii_{k}}=\delta_{z_{k}}\times h_{k,2}\pi_{\sigma_{\ii_{k}}}\mu^{D_{k}}+o(1).
\]

By passing to a subsequence if necessary, we can assume that $(p_{k},q_{k})\rightarrow(p,q)\in\mathcal{F}$,
and that $h_{1,k}\rightarrow h_{1}$ and $h_{2,k}\rightarrow h_{2}$.
Then also $z_{k}\rightarrow z=h_{1}p$ and $h_{2,k}\pi_{\sigma_{\ii_{k}}}\rightarrow h_{2}\pi_{\sigma}$.
Passing to a limit in the previous expression as $k\rightarrow\infty$,
we have 
\[
\widetilde{\mu}=\lim\mu_{k}^{D_{k},\ii_{k}}=\delta_{z}\times h_{2}\pi_{\sigma}\mu
\]
as desired.
\end{proof}

\section{\label{sec:Proof-of-main-result}Proof of Theorem \ref{thm:main}}

\subsection{\label{subsec:Amplifying-dimention-and-proof}Amplifying Assouad
dimension to Hausdorff dimension}

Let $\mathcal{F}\subseteq\mathcal{E}$ be a compact set and let $\delta=\dim_{A}\mathbb{\mathcal{F}}$. Theorem \ref{thm:realizing-Assouad-dim-with-CP}  provides us with a CP-distribution
whose measures arise as weak-* limits of the magnifications of measures supported on $\mathcal{F}$, which are typically $\delta$-dimensional, and, therefore, their supports
are at least as large. But since magnification does not preserve $\mathcal{E}$, typical measures for $P$ also need not be supported
on $\mathcal{E}$, so they give no information about embeddings. 

In this
section we use renormalization instead of magnification to obtain a CP-distribution whose 
measures are supported on $\mathcal{E}$. 

\begin{defn}
A set $\mathcal{F}\subseteq\mathcal{E}$  is said to have \textit{constant scaling} if every $f\in\mathcal{F}$ has the same scaling constant ratio. A measure $\mu\in \mathcal{P}(\mathcal{E})$ is said to have constant scaling if $\text{supp}(\mu)$ has constant scaling.
\end{defn}
 The measures we shall obtain will be  of constant
scaling, which restricts their dimension to be at most $1$. As we shall see later, however, being of
constant scaling is actually an advantage.

Note that homotheties preserve sets of constant scaling, and the representation
of $M_{D,\ii}$ in Proposition \ref{prop:explicit-renormalization-map}
implies that $M_{D,\ii}$ does too. 

Moreover, if $\mathcal{F}\subseteq\mathcal{E}$ has constant scaling
and if $D\in\mathcal{D}_{n}$ with $D\cap\mathcal{E} \neq\emptyset$
then $M_{D,\ii}(\mathcal{F}\cap D)$ and $H_{D}(\mathcal{F}\cap D)$
are homothetic to each other for every suitable $\ii\in I^{*}$. Furthermore, 
the homothety can be taken with coefficients that are bounded in terms
of $\chi_{D}$. This again follows from the form of $M_{D,\ii}$ and the proof of 
Proposition \ref{prop:explicit-renormalization-map}.
\begin{prop}
\label{prop:boosting-assouad-to-hausdorff-on-const-contraction}Let
$\mathcal{F}\subseteq\mathcal{E}$ be a compact set of constant scaling,
and let $\delta=\dim_{A}\mathcal{F}>0$. Then there exists a $\delta$-dimensional
ergodic CP-distribution $P$ such that $P$-a.e. $\mu$ is, up to
a homothety, supported on a constant scaling subset of $\mathcal{E}_0$.
\end{prop}
\begin{proof}
Find a $\delta$-dimensional ergodic CP-distribution $P$ as in
Theorem \ref{thm:realizing-Assouad-dim-with-CP}. Fix a $P$-typical measure $\mu$.  By the same theorem, there is a sequence of dyadic cells $D_{k}\in\mathcal{D}_{n_{k}}$
and probability measures $\mu_{k}$ on $\mathcal{F}\cap D_{k}$, such
that $n_k\to\infty$ and  $\mu_{k}^{D_{k}}\rightarrow\mu$. 

Choose  $\ii_{k}\in I^{k(n_{k})}$ for $k=1,2,3,\ldots$ so that the corresponding renormlization operators $M_{D_k,\ii_k}$ are well defined.  Passing to a subsequence if necessary, we can assume that $\mu_{k}^{D_{k},\ii_{k}}\rightarrow\widetilde{\mu}$
for some measure $\widetilde{\mu}$. Since $\mathcal{F}$ has constant contraction and $\mu_{k}$ are supported on $\mathcal{F}$, also $\mu_{k}$ has constant contraction. This
property is preserved by renormalization, so $\mu_{k}^{D_{k},\ii_{k}}$
has constant contraction. Since this property is also preserved under
limits, $\widetilde{\mu}$ has constant contraction.

Finally, by the formula for $M_{D_{k},\ii_{k}}$ in Proposition \ref{prop:explicit-renormalization-map},
and the fact that $\mathcal{F}$ has constant contraction, there are
maps $H_{k}:\mathbb{R}^{2}\rightarrow\mathbb{R}^{2}$ of bounded norm
whose distance from the set of homotheties tends to $0$ as $k\rightarrow\infty$,
and such that $\mu_{k}^{D_{k},\ii_{k}}=H_{k}\mu_{k}^{D_{k}}$. Passing
to a subsequence if necessary, we can assume $H_{k}\rightarrow H$
for a bona fide homothety $H$, yielding $\widetilde{\mu}=H\mu$,
as claimed.
\end{proof}
Recall the definition of $\mathcal{E}_0$ from \eqref{eq: E zero}.
\begin{thm}
\label{thm:transferring-dim-to-constant-scaling-set}Let $\mathcal{F}\subseteq\mathcal{E}$
be a bounded set, and set $\delta=\dim_{A}\mathcal{F}$ and $\delta^{*}=\min\{1,\delta\}$.
Then there exists a compact subset of constant scaling in $\mathcal{E}_{0}$
with Assouad dimension at least $\delta^{*}$. 

\end{thm}
\begin{proof}
Let $P$ be a $\delta$-dimensional ergodic CP-distribution associated
to $\mathcal{F}$ as in Theorem \ref{thm:realizing-Assouad-dim-with-CP}.  For every $k\in\mathbb{N}$, taking $\varepsilon=1/k$ and large
$N_{k}$, choose dyadic cells $D_{k,1}\supseteq\ldots\supseteq D_{k,N_{k}}$
and a measure $\mu_{k}$ on $D_{k,1}$ as in that theorem, and choose
$x_{k}\in D_{k,N_{k}}\cap\mathcal{F}$ (hence in the intersection
of the cells).

Fix a $P$-typical $\mu$, so that $\dim\supp \mu \geq\dim\mu=\delta$. By Orponen's theorem \ref{Thm Orponen}, with the notation therein, the set of $\sigma\in\mathbb{R}$ such that
$\dim_{A}\pi_{\sigma}\supp \mu<\delta^{*}$ has dimension
zero. 

Choose any measure $\nu$ of positive dimension on $X$ (e.g. any non-trivial self-similar measure, or invoke Frostman's lemma). Positive-dimension of $\nu$ means that $\nu(E)=0$ for every set $E$
of dimension zero. So by the above, 
\begin{equation} \label{eq: def of sigma}
\dim_{A}\pi_{\sigma}\supp \mu \geq\delta^{*}\qquad\text{for \ensuremath{\nu}-a.e. \ensuremath{\sigma}}.
\end{equation}

Fix a $\nu$-typical $\sigma$ satisfying \eqref{eq: def of sigma}. Since $\sigma\in X$ there exists $\ii\in I^{\mathbb{N}}$ such that
$\varphi_{\ii|_{n}}(0)\rightarrow\sigma$ as $n\rightarrow\infty$,
where $\ii|_{n}$ is the $n$-prefix of $\ii$. Write $\sigma_{n}=\varphi_{\ii|_{n}}(0)$. By definition of $P$, we can choose $\widehat{D}_{k}\in\{D_{k,1},\ldots,D_{k,N_{k}}\}$
such that $\mu_{k}^{\widehat{D}_{k}}\rightarrow\mu$. Write $\widehat{D}_k\in\mathcal{D}_{n_{k}}$, and let $\ii_k$ be the prefix of $\ii$ corresponding to $D_k$. We can assume that $n_{k}\rightarrow\infty$. Pass to a subsequence if necessary to ensure that $\mu_{k}^{\widehat{D}_{k},\ii_{k}}$
converges to a measure $\widetilde{\mu}$.

By Proposition \ref{prop:limits-of-renormalized-measures}, $\widetilde{\mu}=\delta_{z}\times h\pi_{\sigma}\mu$
for some $z\in\mathbb{R}$ and $h\in\mathcal{H}$. By the choice of $\sigma$ and \eqref{eq: def of sigma}, we have 
$$\dim_{A}\supp\widetilde{\mu}\geq\delta^{*}.$$ 
Noting that $\supp\widetilde{\mu}\subseteq\mathcal{E}_{0}$
has constant contraction by definition, the proof is complete.
\end{proof}

\subsection{Proof of Theorem \ref{thm:main}}

Our goal is to prove

\begin{thm}\label{thm:technical}
Let $X, Y,\alpha,\beta_j,\Lambda$ be as in the introduction and satisfying the assumptions stated in Section \ref{SSSs-and-assumptions}. If $\Lambda$ satisfies Condition (\oR), then the conclusion of Conjecture \ref{conj:main} holds.
\end{thm}

As explained in the introduction, this implies Theorem \ref{thm:main}.

Suppose by way of contradiction that $\mathcal{E}\neq\emptyset$, that is, that there exists an affine embeddings of $X$ into $Y$.

Write $$\delta:=\sup\{\dim_{A}\mathcal{F}\mid\mathcal{F}\subseteq\mathcal{E}\text{ is compact}\}.$$

\begin{lem}
$\delta>0$.
\end{lem}
\begin{proof}
Let $\pi:\mathcal{E}\rightarrow\mathbb{R}$ be map $f\mapsto\log\left\Vert f\right\Vert $.
By \cite{FengXiong2018} (see also Lemma \ref{lem: 1} and Proposition \ref{prop:renomalization-at-point-gives-alpha-beta-set}),
there exists a compact $\mathcal{F}\subseteq\mathcal{E}$
such that 
\[
\overline{\dim}_{B}\pi\mathcal{F}>0.
\]
Since $\pi$ is Lipschitz on Compact sets, we conclude that $\overline{\dim}_{B}\mathcal{\mathcal{F}}>0$, and so
 $\dim_{A}\mathcal{F}>0$.
\end{proof}

Write
\[
\delta^{*}=\min\{1,\delta\}.
\]
\begin{lem} \label{Lemma P}
There exists a $\delta^{*}$-dimensional CP-distribution
$P$ whose measures, up to a homothety taken from a compact set, are supported on $\mathcal{E}_{0}$ and have constant scaling. In particular, $\dim_{A} \mathcal{E}_{0} = \delta<1$.
\end{lem}
\begin{proof}
It follows from Theorem \ref{thm:transferring-dim-to-constant-scaling-set} that we have $\dim_{A}\mathcal{E}_{0}\ge \delta^*$.
Since $\mathcal{E}_{0}$ is itself compact,  we may apply Theorem \ref{thm:transferring-dim-to-constant-scaling-set} 
to $\mathcal{E}_{0}$ to find a set $\mathcal{G}\subseteq\mathcal{E}_{0}$
of constant contraction with Assouad dimension $\delta^{*}$. Now, 
apply  Proposition \ref{prop:boosting-assouad-to-hausdorff-on-const-contraction} to $\mathcal{G}$,  to find the required $\delta^{*}$-dimensional CP-distribution.

If $\delta^{*}=1$ then $P$-a.e. measure $\mu$ is supported on a set of the form $\lbrace z \rbrace \times T$, where $\dim T=1$ (actually, $T$ can be shown to be an interval using Theorem \ref{thm:dim-of-CP}, but we don't need this). This is impossible since
then, for any $x_{0}\in X\setminus\{0\}$, the set $z\cdot x_0+T$ has dimension $1$ as an affine image of the set $T$, but is also a subset of $Y$. This  contradicts the strong separation assumption on $Y$.  Thus we must have $\delta^{*}=\delta<1$.  Since $\mathcal{E}_{0}$ is  compact,  by definition of $\delta$,  we have $\dim_{A}\mathcal{E}_{0}\le \delta$.  We conclude that $\dim_{A} \mathcal{E}_{0} = \delta$.
\end{proof}

We aim to derive a contradiction by showing that $\overline{\dim}_{B}\mathcal{E}_{0}>\delta$, which, by Lemma \ref{Lemma P}, cannot happen. The argument is
based on the following elementary observation:
\begin{lem} \label{Lemma covering}
Let $Z\subseteq\mathcal{E}$ and let $\pi:Z\rightarrow\mathbb{R}$
denote projection to the first coordinate. For $p\in Z$ write $Z_{p}=\pi^{-1}(\lbrace \pi(p) \rbrace)=\{q\in Z\mid\pi(p)=\pi(q)\}$.
Suppose that for some $\varepsilon_{k}\rightarrow0$ and $\gamma>0$,
\[
\cov(Z_{p},1/k)\geq k^{\gamma-\varepsilon_{k}}\qquad\text{for all }p\in Z\text{ and }k\geq0.
\]
Then
\[
\overline{\dim}_{B}Z\geq\gamma+\overline{\dim}_{B}\pi(Z)
\]
\end{lem}
Lemma \ref{Lemma covering} leads to the desired contradiction via the following argument. Let $P$ be the CP-distribution from Lemma \ref{Lemma P}, let $\mu$ be a $P$-generic measure (in the sense of say Theorem \ref{thm:CP-equidistribution}), and let $f$ be a $\mu$-typical map. Note that   $f\in \mathcal{E}_{0}$, and write $D_{n}=\mathcal{D}_{n}(f)$.

By Theorem \ref{thm:CP-equidistribution}, the sequence $\mu^{D_{n}}$
equidistributes for $P$.  Let $\ii \in I^{\mathbb{N}}$ be any fixed sequence, and let $\ii_{n}$ denote the appropriate sub-words so that renormlization is well defined. Then the measures $\mu^{D_{n},\ii_{n}}$
are equal to $\mu^{D_{n}}$ up to homotheties bounded in terms of
$f$ (by e.g. Proposition \ref{prop:explicit-renormalization-map}). But also, by Proposition \ref{prop:renomalization-at-point-gives-alpha-beta-set},
on a syndetic subsequence $n_{i}$ of $n$, the logarithms $\theta_{n}$
of the scaling constants corresponding to $\mu_{n}^{D_{n},\ii_{n}}$
form a syndetic subsequence of an $\Lambda$-multi-rotation, with
$\Lambda=\{\log\beta_{j}/\log\alpha\}_{j\in J}$.

We now claim that
\begin{equation} \label{eq: box dim increase}
\overline{\dim}_{B}(\bigcup_{i}\supp\mu^{D_{n_{i}},\ii_{n_{i}}})\geq\delta+\varepsilon \text{ for some } \varepsilon>0.
\end{equation} 
We will use Lemma \ref{Lemma covering} to do this, where $\delta$ will play the role of $\gamma$. 

Since $P$-a.e. $\nu$ has dimension $\delta$, the support of $P$
-a.e. such $\nu$ has box dimension at least $\delta$. It follows
that there exists a set $\Omega$ of $P$-measure arbitrarily close
to $1$, and $\varepsilon_{k}\rightarrow0$, such that 
\[
\cov(\supp\nu,1/k)>(1/k)^{-(\delta-\varepsilon_{k})}\qquad\text{for all \ensuremath{\nu\in\Omega}}, k\in \mathbb{N}.
\]
Sine $\mu$ is $P$-generic, we may assume the set
\[
U=\{n\in\mathbb{N}\mid\mu^{D_{n}}\in\Omega\}
\]
has density arbitrarily close to $1$.

It follows from syndeticity of $n_{i}$ that $\{n_{i}\}\cap U$ has
positive density. Therefore by Condition (\oR) on $\Lambda$, 
$$\varepsilon:=\overline{\dim}_{B}\{\theta_{n_{i}}\mid n_{i}\in U\}>0.$$
By Lemma \ref{Lemma covering}, 
$$\overline{\dim}_{B} \left( \bigcup_{i\mid n_{i}\in U}\supp\mu^{D_{n_{i}},\ii_{n_{i}}} \right)\geq\delta+\varepsilon.$$
This prove \eqref{eq: box dim increase}, which, combined with Lemma \ref{Lemma P}, is the desired contradiction.

\section{\label{sec:remaining_proofs}Proof of Theorems \ref{thm:algebraic}, \ref{thm:nonhom} and \ref{thm:generic}}

\subsection{Preliminaries}

Let $X=\bigcup_{i\in I}\varphi_{i}X$ and $Y=\bigcup_{j\in J}\psi_{j}(Y)$ 
be self-similar sets as before and assume for notational convenience that $I=\{1,\ldots,M\}$,
$J=\{1,\ldots,L\}$. For non-trivial $\uu\in I^{*}$ let $n(\uu)\in\mathbb{Z}^{M}$
denote the non-negative, non-zero vector which counts the number of
occurrences of each symbol: 
\[
n(\uu)_{i}=\#\text{ of appearences of \ensuremath{i} in \ensuremath{\uu}}.
\]
Thus $\left\Vert \varphi_{\uu}\right\Vert =\prod_{i\in I}\alpha_{i}^{n(\uu)_{i}}$
and $\log\left\Vert \varphi_{\uu}\right\Vert =\sum_{i\in I}n(\uu)_{i}\log\alpha_{i}$.
\begin{lem}
\label{lem:sub-IFSs}Suppose that $\log\alpha_{r}\notin\spn_{\mathbb{Q}}\{\log\beta_{j}\}$
for some $r\in I$. Then there exist $\uu^{(1)},\ldots,\uu^{(M)}\in I^{*}$ and $\vv^{(1)},\ldots,\vv^{(M)}\in I^{*}$
such that:
\begin{enumerate}
\item $n(\uu^{(1)}),\ldots,n(\uu^{(M)})$ are linearly independent over $\mathbb{Q}$.
\item $\log\left\Vert \varphi_{\uu^{(i)}}\right\Vert=\log\left\Vert \varphi_{\vv^{(i)}}\right\Vert \notin\spn_{\mathbb{Q}}\{\log\beta_{j}\}$.
\item $\varphi_{\uu^{(i)}},\varphi_{\vv^{(i)}}$ have no common fixed point.
\end{enumerate}
\end{lem}

\begin{proof}
Let $r\in I$ be as in the statement and let $p,q\in I$ be such that
$\varphi_{p},\varphi_{q}$ don't commute, or equivalently, have different fixed points.\footnote{Under our assumption of strong separation any distinct $p,q$ will do here} Fix $i$ and for a large $N$, consider the
words $\uu^{(i,1)}=(pq)i^{N}r$ and $\uu^{(i,2)}=(pq)i^{N}r^{2}$.
At least one of $\log\left\Vert \varphi_{\uu^{(i,1)}}\right\Vert $,
$\log\left\Vert \varphi_{\uu^{(i,1)}}\right\Vert $ does not belong
to $\spn_{\mathbb{Q}}\{\log\beta_{j}\}$. Indeed, otherwise their difference
$\log\alpha_{r}$ would belong to this space, contrary to assumption.
Define $\uu^{(i)}\in\{\uu^{(i,1)},\uu^{(i,2)}\}$ with this
property. This gives the part of (2) relating to $\uu$, and it is easy to check that $n(\uu^{(i)})$
satisfy (1) if $N$ is large enough, which we assume. Finally, given $i\in I$, let $\vv^{(i)}$ be the same
as $\uu^{(i)}$ but with the initial $pq$ replaced by $qp$. 
The scaling constants of $\varphi_{\uu^{(i)}},\varphi_{\vv^{(i)}}$
are the same, so (2) holds in full, and by choice of $p,q$ they have no common fixed point, which is (3).
\end{proof}
We shall refer to two Diophantine conditions:
\begin{defn}
A real sequence $\gamma_{1},\ldots,\gamma_{k}\in\mathbb{R}$ satisfies 
\begin{itemize}
\item Condition (D), if there exists $c>0$ such that for all $2\leq N\in\mathbb{N}$
and every choice of integers $-N\leq n_{i}\leq N$ that are not all
zero, $|\sum n_{i}\gamma_{i}|>N^{-c}$. 
\item Condition (d), if there exists $c>0$ such that for infinitely many 
$N\in\mathbb{N}$ and every choice of integers $0\leq n_{i}\leq N$
that are not all zero, $|\sum n_{i}\gamma_{i}|>N^{-c}$. 
\end{itemize}
\end{defn}

\begin{rem} Notice that
\begin{enumerate}
\item Condition (D) is stronger than (d), as it allows combinations
of $\gamma_{i}$ using coefficients of mixed signs, and requires the inequality to hold for all $N\geq 2$ rather than arbitrarily large $N$.
\item Condition (D) implies that $\gamma_1,\ldots\gamma_M$ are  $\mathbb{Q}$-linearly independent. This is not implied by (d) since it could be that all non-trivial vanishing linear combinations have some negative coefficients. For the same reason, repetitions of elements in the sequence $\gamma_i$ are irrelevant, so we can apply Condition (d) to sequences or to the set of elements of a sequence interchangeably.
\item  Note that in this Definition we have not reduced mod 1.  

\end{enumerate}
\end{rem}
The following Lemma relates Condition (d) to Condition (\uR).
\begin{lem} \label{lem:d-implies-R}
Let $\Lambda=\Lambda(X,Y)$ be defined in the setting of Theorem \ref{thm:main}. If $\{-1\}\cup\Lambda$ satisfies Condition (d) (equivalently,  if $\log\alpha,-\log\beta_1,\ldots,-\log\beta_L$  satisfy Condition (d)), then $\Lambda$ satisfies Condition (\uR).
\end{lem}
\begin{proof}
Write $\sigma_j=\log\beta_j/\log\alpha$ so that $\Lambda=\{\sigma_j\}$ and let $d(\cdot,\cdot)$ denote distance modulo one. Let $\sigma=\max\{1,\sigma_1,\ldots\sigma_L\}$. If $(\theta_n)_{n=1}^\infty$ is any $\Lambda$-multi-rotation, then for every $k>\ell$ there are integers $0\leq n_1,\ldots,n_L\leq k-\ell$ and $0\leq n \leq \sigma(k-\ell)$ such that \[
   d(\theta_k,\theta_\ell) =|-n+\sum_{j=1}^L n_j\sigma_j|,\]
where the $-n$ term reduces modulo one. Therefore, if $-1,\sigma_1,\ldots,\sigma_L$ satisfy Condition (d), then for infinitely many $N\in \mathbb{N}$, the distance between any two terms in $\theta_1,\ldots,\theta_N$ is at least $(\sigma N)^{-c}$. Then the $(\sigma N)^{-c}$-covering number of any $\varepsilon$-fraction of these points is $\varepsilon N$, showing that the upper box dimension of any positive-density subsequence of $(\theta_n)$ is at least $1/c$.
\end{proof}

We note that, if we modified Condition (d) to require the inequality for every $N$ rather than arbitrarily large $N$, then the argument in the lemma above would give the stronger Condition (\uR).

As is evident from the remarks above, the signs of the numbers we work with will be significant. To avoid confusion, we work with non-negative parameters where possible, and in particular we adopt the convention that $\gamma_i,\lambda_j$ always represent non-negative numbers. In our notation for self-similar sets we will use the dictionary  $\gamma_i=-\log\alpha_i$ and $\lambda_j=-\log\beta_j$.

\subsection{Proof of Theorem \ref{thm:algebraic}}

In our terminology, Baker's theorem can be stated as follows \cite{Baker1972}:
\begin{thm}
If $\sigma_{1},\ldots,\sigma_{M}$ are algebraic and $\log\sigma_{1},\ldots,\log\sigma_{M}$
are $\mathbb{Q}$-linearly independent, then $\{\log\sigma_{1},\ldots,\log\sigma_{M}\}$
satisfies Condition (D).
\end{thm}

Theorem \ref{thm:algebraic} in now proved as follows.  Suppose that $X,  Y,\{\alpha_{i}\}_{i=1}^M,\{\beta_{j}\}_{j=1}^L$
are as in the statement of Theorem \ref{thm:algebraic}, so all $\alpha_{i},\beta_{j}$
are algebraic and $\log\alpha_{i}\notin\spn_{\mathbb{Q}}\{\beta_{j}\}$
for some $i$. We want to show that $\mathcal{E}=\emptyset$.

By Lemma \ref{lem:sub-IFSs}, we can assume that $X$ is homogeneous,
i.e.   $\alpha_{i}=\alpha$ for all $i$. For otherwise, let $\uu^{(1)},\vv^{(1)}$
be as in the lemma, let $X'\subseteq X$ be the self-similar set defined
by $\varphi_{\uu^{(1)}},\varphi_{\vv^{(1)}}$, and note that the scaling
constants of these maps are products of the original $\alpha_{1},\ldots,\alpha_{M}$,
so they are also algebraic. Thus if we knew the statement for homogeneous
$X$ with algebraic scaling constants we would deduce that $\mathcal{E}(X,Y)\subseteq\mathcal{E}(X',Y)=\emptyset$.

Thus, assume homogeneity. Set $\gamma=-\log\alpha$ and $\lambda_j=-\log\beta_j$, so we are assuming also $\gamma\notin\spn_{\mathbb{Q}}\{\lambda_{j}\}$.
To derive $\mathcal{E}=\emptyset$, we wish to prove that Condition
(\uR) holds. By Lemma \ref{lem:d-implies-R}, it suffices to show that $\{-\gamma,\lambda_{1},\ldots,\lambda_{L}\}$
satisfies Condition (d). This would follow directly from Baker's theorem
if $\{-\gamma,\lambda_{1},\ldots,\lambda_{L}\}$ were $\mathbb{Q}$-linearly
independent, but with our assumptions, we cannot rule out linear relations
between the $\lambda_{j}$. To remedy this, choose a maximal linearly independent subset of the  
 $\lambda_{j}$'s, and note that if $|-n\gamma+\sum_{j=1}^{L}n_{j}\lambda_{j}|$ is sufficiently
small and $0\leq n,n_{j}\leq N$ not all zero, by sign considerations we
must have $n\neq0$. Then we can re-write $\sum_{j=1}^{L}n_{j}\lambda_{j}$
in terms of the chosen basis. Now, even if this sum vanishes, we still
get a non-trivial (because $n\neq0$) linear combination of $\gamma,\lambda_{1},\ldots,\lambda_{L}$,
with rational coefficients whose numerator and denominator are in
the range $[-CN^{2},CN^{2}]$ for some $C$ that depends only on the
$\lambda_{j}$. Baker's bound for the new expression implies it for the
original one, possibly with a larger exponent.  This proves that condition (d) holds, and so by Lemma \ref{lem:d-implies-R} we have Condition (\uR); The result now follows from Theorem \ref{thm:main}.

\subsection{Proof of Theorem \ref{thm:nonhom}}
The following Lemma is the key for deriving Theorem \ref{thm:nonhom} from Theorem \ref{thm:main}.
\begin{lem} \label{Lemma nonhom}
Let $\gamma_{1},\ldots,\gamma_{M}>0$ satisfy Condition
(D). If $\rank_{\mathbb{Q}}\{\gamma_{i}\}_{i=1}^M>L$, then for every
$\lambda_{1},\ldots,\lambda_{L}>0$ there exists $i\in\{1,\ldots,M\}$
such that $-\gamma_{i},\lambda_{1}\ldots\lambda_{L}$ satisfy Condition
(d).
\end{lem}

\begin{proof}
For $x,y\in\mathbb{R}^{L}$ write $\pi_{y}(x)=\left\langle x,y\right\rangle $,
so $\pi_{y}:\mathbb{R}^{L}\rightarrow\mathbb{R}$.
Let $c$ be the constant verifying Condition (D) for the $\gamma_{i}$.
For $N\in\mathbb{N}$ set $\delta_{N}=N^{-(L+1)(c+1)}$, and 
\begin{align*}
\Omega_{N,i} & =\{\lambda\in(0,\infty)^{L}\mid\exists0\leq k,n_{j}\leq N\text{ not all zero with }|-k\gamma_{i}+\sum_{j=1}^{L}n_{j}\lambda_{j}|<\delta_{N}\}\\
 & =(0,\infty)^{L}\cap\bigcup_{0\neq n\in([0,N]\cap\mathbb{Z})^{L}}\bigcup_{k=1}^{N}\pi_{n}^{-1}(B_{\delta_{N}}(-\gamma_{i}k)).
\end{align*}
The last equality again uses
considerations of signs: Any $k,n_{1},\ldots,n_{L}$ satisfying the condition
in the first line, must satisfy $k\neq0$ and $(n_{1},\ldots,n_{L})\neq0$.

It suffices for us to show that $\bigcap_{i=1}^{M}\Omega_{N,i}=\emptyset$
for every $N\in\mathbb{N}$, since this means that for every $\lambda\in(0,1)^{L}$
and every $N$ there is an $i$ such that $\lambda\notin\Omega_{N,i}$,
and hence there is some $i$ such that $\lambda\notin\Omega_{N,i}$
for infinitely many $N$. This implies Condition (d) with exponent $(L+1)(c+1)$.

Fix $N$. Then 
\[
\bigcap_{i=1}^{M}\Omega_{N,i}\subseteq \bigcup_{0\neq n^{(1)},\ldots,n^{(M)}\in([0,N]\cap\mathbb{Z})^{L}}\bigcup_{k_{1},\ldots,k_{M}=1}^{N}\left(\bigcap_{i=1}^{M}\pi_{n^{(i)}}^{-1}(B_{\delta_{N}}(-\gamma_{i}k_{i}))\right).
\]
We must show that each of the inner intersections on the right is
empty. To this end, fix $n^{(1)},\ldots,n^{(M)}$ and $ k_{1},\ldots,k_{M}$
and suppose by way of contradiciton that 
\begin{align*}
\lambda & \in\bigcap_{i=1}^{M}\pi_{n^{(i)}}^{-1}(B_{\delta_{N}}(-\gamma_{i}k_{i})).
\end{align*}
Then there exist $0<|\varepsilon_{i}|<\delta_{N}$, for $i=1,\ldots,M$,
such that
\[
\left\langle n^{(i)},\lambda\right\rangle =-\gamma_{i}k_{i}+\varepsilon_{i}.
\]
Since $M>L$, one of the $n^{(i)}$, say $n^{(M)}$, is a rational
linear combination of the others, which we can write as $\ell_{M}n^{(M)}=\sum_{i=1}^{M-1}\ell_{i}n^{(i)}$
for integers $\ell_{i}$. Since $n^{(i)}$ all have coefficients of
order $N$, the coefficients $\ell_{i}$ are of order $O(N^{L})$.
We have
\begin{align*}
\ell_M(-\gamma_{M}k_{M}+\varepsilon_{M} )& =\left\langle \ell_Mn^{(M)},\lambda\right\rangle \\
 & =\left\langle \sum_{i=1}^{M-1}\ell_{i}n^{(i)},\lambda\right\rangle \\
 & =\sum_{i=1}^{M-1}\ell_{i}k_{i}\gamma_{i}+\sum_{i=1}^{M-1}\ell_{i}\varepsilon_{i}.
\end{align*}
Using the definition of $\delta_{N}$ and $|\varepsilon_{i}|<\delta_{N}$,
the identity above yields
\[
\ell_Mk_{M}\gamma_{M}+\sum_{i=1}^{M-1}\ell_{i}k_{i}\gamma_{i}=O_{M}(N^{L}\delta_{N})=O_{M}(N^{L}\cdot N^{-(L+1)(c+1)})=o((N^{L+1})^{c}).
\]
The left hand side is a non-trivial integer combination of
$\gamma_{1},\ldots,\gamma_{M}$ with coefficients of order $O(N^{L+1})$,
so we have a contradiction to the hypothesis that $\{\gamma_{i}\}$
satisfies Condition (D) with exponent $c$.
\end{proof}
We can now prove Theorem \ref{thm:main}. Write $X=\bigcup_{i=1}^{M}\varphi_{i}(X),Y=\bigcup_{j=1}^{L}\psi_{j}(Y)$
and suppose that $\alpha_{i}$ are algebraic and satisfy $\rank_{\mathbb{Q}}\{\log\alpha_{i}\}_{i\in I}>L$.
We shall show that $\mathcal{E}(X,Y)=\emptyset$.

We can assume without loss of generality that $M=\rank_{\mathbb{Q}}\{\log\alpha_{i}\}_{i\in I}$.
Otherwise, remove $\varphi_i$ corresponding to redundant $\alpha_{i}$, and restrict discussion to the resulting
self-similar subset of $X$. If this subset does not embed in $Y$
then neither does $X$.

Let $\uu^{(i)},\vv^{(i)}$, $i=1,\ldots,M$, be as in Lemma \ref{lem:sub-IFSs}.
Set $\alpha'_{i}=\left\Vert \varphi_{\uu^{(i)}}\right\Vert $. These
are products of algebraic numbers so they are algebraic, and by part (1) of the Lemma, the $\gamma_{i}=-\log\alpha'_{i}$
are $\mathbb{Q}$-linearly independent. By Baker's Theorem the $\gamma_{i}$
satisfy Property (D).

Now let $\beta_{1},\ldots,\beta_{L}\in(0,1)$ and $\lambda_j=-\log\beta_j$. By the previous
lemma we can find $i$ such that $\gamma_{i},\lambda_{1},\ldots,\lambda_{L}$
satisfy property (d). Therefore, for the homogeneous self-similar
set $X'$ defined by $\varphi_{\uu^{(i)}},\varphi_{\vv^{(i)}}$ and
$Y$ defined by any maps with scaling constants $\beta_{1},\ldots,\beta_{L}$,
the corresponding set $\Lambda$ satisfies Condition (\uR) by Lemma \ref{lem:d-implies-R}. By Theorem \ref{thm:main} we now have
$\mathcal{E}(X,Y)\subseteq\mathcal{E}(X',Y)=\emptyset$, as claimed.

\subsection{Proof of Theorem \ref{thm:generic}}

For this section we will use the notion of packing dimension (also known as modified box dimension) of a
set $E$:
\[
\dim_{P}E=\inf\{\sup_{n}\overline{\dim}_{B}E_{n}\mid E=\bigcup_{n=1}^{\infty}E_{n}.\}
\]
See e.g. \cite{mattila1999geometry} for a discussion of this notion. We will require the following standard properties:
\begin{itemize}
\item If $E_{n}\subseteq\mathbb{R}$ and $E=\liminf E_{n}=\bigcup_{n}\bigcap_{k\geq n}E_{k}$,
then for any sequence $r_{k}\rightarrow0$ with $\limsup \frac{\log r_{k+1}}{\log r_k}=1$,
\begin{equation}
\dim_{P}E\leq\sup_{n}\limsup_{k\rightarrow\infty}\frac{\log\cov(E_{k},r_{k})}{\log(1/r_{k})}\label{eq:pdim-of-liminf}
\end{equation}

\item  For any two Borel sets, \begin{equation}
\dim_{P}E\times F\leq\dim_{P}E+\dim_{P}F\label{eq:product-dim}
\end{equation}
\end{itemize}

\subsubsection*{Proof of Part (1)}

\begin{lem}\label{lem:lambdas-fixed}
Fix $\lambda_{1},\ldots,\lambda_{L}>0$. Then the set
\[
E=\{\gamma>0\mid\{-\gamma,\lambda_{1},\ldots,\lambda_{L}\}\text{ does not satisfy Condition (d)}\}
\]
has Packing dimension zero.
\end{lem}

\begin{proof}
Fix $T>0$. We shall show that $\dim_{P}E\cap(0,T)=0$ and hence
$\dim_{P}E=0$.

By considering signs, any linear combination of $\gamma,\lambda_{1},\ldots,\lambda_{L}$
with non-negative coefficients that come close to zero must give $\gamma$
and at least one $\lambda_{j}$ non-zero coefficients. Thus, $E\cap(0,T)=\bigcap_{c=1}^{\infty}E_{c}$,
where
\begin{align*}
E_{c} & =\liminf_{N\rightarrow\infty}\left(\bigcup_{0\neq n\in([0,N]\cap\mathbb{Z})^{L}}\bigcup_{k=1}^{N}B_{N^{-c}/k}(\sum_{j=1}^{L}\frac{n_{j}}{k}\lambda_{j}))\cap(0,T)\right)
\end{align*}
Each term in the liminf is the union of $O(N^{L+1})$ balls of radius
at most $N^{-c}$. By \eqref{eq:pdim-of-liminf} this shows  that $\dim_{P}(E_{c}\cap(0,T))\leq(L+1)/c$,
so 
\[
\dim(E\cap(0,T))\leq\liminf_{c\rightarrow\infty}\dim E_{c}\leq\lim_{c\rightarrow\infty}\frac{L+1}{c}=0
\]
as claimed.
\end{proof}
Part (1) of Theorem \ref{thm:generic} can now be proved as follows.
Let $\beta_{1},\ldots,\beta_{L}\in(0,1)$, and set $\lambda_{j}=-\log\beta_{j}$.
Let $E\subseteq\mathbb{R}$ denote the set defined in Lemma \ref{lem:lambdas-fixed} 
for these $\lambda_{j}$'s.

Denote by $F$ the set of $(\alpha_{1},\ldots,\alpha_{M})\in(0,1)^{M}$
for which the conjecture fails. So $X$ embeds in $Y$, but one of
the $\log\alpha_{i}$'s is not in the $\mathbb{Q}$-span of the $\lambda_{j}$'s. 

Apply Lemma \ref{lem:sub-IFSs} to obtain $\uu^{(i)},\vv^{(i)}$ as
in the lemma. Write $\gamma_{i}=-\log\left\Vert \varphi_{\uu^{(i)}}\right\Vert $
and $n^{(i)}=n(\uu^{(i)})$. Let $X_{i}\subseteq X$ be the homogeneous
self-similar set\footnote{The reason we pass to words $\uu^{(i)}$ and $\gamma_i$ instead of working directly with the original maps and with $-\log\alpha_i$, is that knowing that $-\alpha_i\in E$ does not give us a homogeneous self-similar set in X with contraction $\alpha_i$. To get such a set we need at least two maps with contraction $\alpha_i$. That is what the pairs $\uu^{(i)},\vv^{(i)}$ give us.} defined by $\varphi_{\uu^{(i)}},\varphi_{\vv^{(i)}}$. 

Since $X$ embeds in $Y$, all of the $X_{i}$ embed in $Y$, so the
corresponding sets $\Lambda_{i}$ cannot satisfy Condition (\uR).
In particular, by Lemma \ref{lem:d-implies-R} $-\gamma_{i},\lambda_{1},\ldots,\lambda_{L}$ do not
satisfy Condition (d). Thus $\gamma_{i}\in E$. 

For $\alpha\in F$ write $\log\alpha=(\log\alpha_i)_{i=1}^M$ and $\log(F)=\{\log\alpha\mid\alpha\in F\}$. It follows that 
\[
\log F\subseteq\bigcap_{i=1}^{M}\pi_{n^{(i)}}^{-1}(E)
\]
writing $\pi(x)=(\pi_{n^{(1)}}x,\ldots,\pi_{n^{(M)}}x)$ this says
\[
\log F\subseteq\pi^{-1}(\prod_{i=1}^M E)
\]
Since the $n^{(i)}$ are linearly independent, $\pi$ is a linear
bijection of $\mathbb{R}^{M}$, so 
\[
\dim_{P}\log F\leq\dim_{P}E^{M}\leq\sum_{i=1}^{M}\dim_{P}E=0.
\]
We used \eqref{eq:product-dim} to bound the dimension of $\prod_{i=1}^M E$. Finally, $\log$ is locally bi-Lipschitz, so it does not change packing
dimension, proving the theorem.

\subsubsection{Proof of Part (2)}
\begin{lem}
Fix $\gamma>0$. Then for every $L\in\mathbb{N},$the set
\[
E=\{\lambda\in(0,\infty)^{L}\mid-1,\frac{\lambda_{1}}{\gamma},\ldots,\cdot\frac{\lambda_{L}}{\gamma}\text{ do not satisfy Condition (d)}\}
\]
has packing dimension at most $L-1$.
\end{lem}

\begin{proof}
Fix $T>0$. We shall show that $\dim_{P}E\cap(0,T)\leq L-1$ and hence $\dim_{P}E\leq L-1$. We have $E\cap(0,T)^L=\bigcap_{c=1}^{\infty}E_{c}$, where
\begin{align*}
E_{c} & =\liminf_{N\rightarrow\infty}\left(\bigcup_{0\neq n\in([0,N]\cap\mathbb{Z})^{L}}\bigcup_{k=1}^{N}\pi_{n/\gamma}^{-1}(B_{N^{-c}}(k))\cap(0,T)^{L}\right)
\end{align*}
Now, for any $\delta>0$, the set $\pi_{n/\gamma}^{-1}(B_{\delta}(k))$
is the $\gamma\delta/\left\Vert n\right\Vert $-neighborhood of
an affine subspace normal to $n$, so $\pi_{n/\gamma}^{-1}(B_{\delta}(k))\cap(0,T)^{L}$
can be covered by $O_{\gamma,T}(\delta^{-L})$ balls of radius $\delta$.
Noting that the set in the liminf above is the union of $(2N+1)^{L+1}$
such sets, we see that each term in the liminf can be covered by $O(N^{L+1}\cdot(N^{-c})^{L-1})=O(N^{-c(L-1-L/c)})$
balls of radius $N^{-c}$.  By \eqref{eq:pdim-of-liminf} this shows  that $\dim_{P}E_{c}\leq L-1+L/c$.
Thus 
\[
\dim_{P}E=\lim_{c\rightarrow\infty}\dim_{P}E_{c}\leq L-1
\]
as claimed.
\end{proof}
Part (2) of Theorem \ref{thm:generic} follows directly from the Lemma
using the dictionary $\gamma=-\log\alpha$ and $\lambda_{j}=-\log\beta_{j}$ and Lemma \ref{lem:d-implies-R}.

\bibliographystyle{plain}
\bibliography{bib}

\bigskip
\bigskip
\footnotesize

\noindent{}\texttt{Department of Mathematics, University of Haifa at Oranim, Tivon 36006, Israel\\ email: amir.algom@math.haifa.ac.il}
\newline

\noindent{}\texttt{Department of Mathematics, The Hebrew University of Jerusalem, Jerusalem, Israel\  and Warwick University, UK\\ email: michael.hochman@mail.huji.ac.il}
\newline

\noindent{}\texttt{Department of Mathematical Sciences, P.O. Box 3000, 90014 University of Oulu, Finland\\email:  meng.wu@oulu.fi}

\end{document}